\documentclass[reqno]{amsart}
\usepackage{amsfonts}
\usepackage{amsmath}
\usepackage{amsthm, amscd}
\usepackage{mathtools}
\usepackage{amssymb}
\usepackage{mathrsfs}
\usepackage{graphicx}
\usepackage{caption}
\usepackage{comment}
\usepackage{subcaption}\usepackage{tikz-cd}
\usepackage{float}
\usepackage{xypic}
\usepackage[abs]{overpic}
\usepackage[alphabetic]{amsrefs}
\usepackage{etex}
\usepackage{tikz-cd}

\usepackage{hyperref}
\hypersetup{
	colorlinks,
	linkcolor=[rgb]{0,0,0.7},
	urlcolor=[rgb]{0,0,0.4},
	citecolor=[rgb]{0.4,0.1,0}
}

\allowdisplaybreaks

\theoremstyle{plain}\newtheorem{Theorem}{Theorem}[section]
\theoremstyle{plain}\newtheorem{Corollary}[Theorem]{Corollary}
\theoremstyle{plain}\newtheorem{Lemma}[Theorem]{Lemma}
\theoremstyle{plain}
\theoremstyle{plain}
\theoremstyle{plain}
\theoremstyle{plain}
\newtheorem{lemma}[Theorem]{Lemma}
\theoremstyle{plain}\newtheorem*{Claim*}{Claim}
\theoremstyle{plain}\newtheorem*{Theorem*}{Theorem}
\theoremstyle{plain}

\theoremstyle{remark}\newtheorem{remark}[Theorem]{Remark}
\theoremstyle{remark}
\theoremstyle{remark}\newtheorem*{Notation*}{Notation}

\theoremstyle{plain}
\makeatletter
\newtheorem*{rep@theorem}{\rep@title}
\newcommand{\newreptheorem}[2]{
\newenvironment{rep#1}[1]{
 \def\rep@title{#2 \ref{##1}}
 \begin{rep@theorem}}
 {\end{rep@theorem}}}
\makeatother
\newreptheorem{Theorem}{Theorem}
\newreptheorem{Proposition}{Proposition}
\newreptheorem{Corollary}{Corollary}

\numberwithin{equation}{section}

\usepackage{lipsum}

\DeclareMathOperator{\id}{id}

\DeclareMathOperator{\Dax}{Dax}

\author{Huizheng Guo}
\address{Institute of Geometry and Physics, University of Science and Technology of China}
\email{aliguo@ustc.edu.cn}

\title{Every $S^2$-Bundle over a positive genus($g>1$) Surface has Mapping Class Group of Infinite Rank}

\begin{document}

\maketitle

\begin{abstract}
    In 2025, Lin, Wu, Xie and Zhang extend Dax invariant from embedded disks to embedded surfaces and show that there is surjective homomorphism from $MCG(M)$ to $\mathbb{Z}^{\infty}$ for $M= \Sigma \times S^2$.  Let $\Sigma$ be a closed surface with positive$(g>1)$ genus, in the paper we show that for any nontrivial $S^2$ fiber bundles $X= \Sigma \ltimes S^2$, whether $X$ orientable or nonorientable, there is a surjective homomorphism from $MCG(X)$ to $\mathbb{Z}^{\infty}$. The proof is an application of  generalization of Dax invariants for embedded surfaces and the restriction on pseudo-isotopy subgroup $MCG_0(X)$ also surjects on $\mathbb{Z}^{\infty}$. The property of $MCG(X)$ inherits from trivial fiber bundle $\Sigma \times S^2$ and main theorem shows that the algebraic size of the smooth mapping class group of $S^2$-fiber bundles does not depend on the triviality of the bundle. 
\end{abstract}
\vspace{-0.15cm}
\section{Introduction}

Given a closed orientable surface $\Sigma$ with finite genus and given a closed nonorientable surface $\Sigma'$ with finite genus. Let $M$ be a $S^2$ fiber bundle over a closed surface, $MCG(M)$ is defined to be the smooth mapping class group i.e. the group of isotopy classes of orientation preserving diffeomorphism. Alternatively, we say that $MCG(M)$ are path components of $Diff^+(M)$, i.e. $MCG(M) = \pi_0(Diff^+(M))$. Let $MCG_0 \subset MCG(M)$ consists of diffeomorphisms that are pseudo isotopic to identity. That is, there is a mapping cylinder $f: M\times I \rightarrow M\times I$ s.t. $f|_{M\times \{0\}} = id$. Consider the forgetting map $h: MCG(M) \rightarrow Aut(M)$ where $Aut(M)$ denote the homotopy equivalence between $M$ and itself up to homotopy, then $ker(h) = MCG_0(M)$, where  $$MCG_0(M) = \{[f]\in MCG(M) | f  \ \textit{is pseudo-isotopic to} \ id\}$$.

The computation and properties of mapping class group of 4-manifold in smooth category have continuously being a popular exploration in 4-manifold. There are indeed more results in topological category: closed simple connected 4-manifold were shown by combination work of Perron\cite{P86}, Quinn\cite{Q86} and Freedman\cite{F82} in the approach of intersection form. The nontrivial boundary 4-manifold case was recently extended by Orson and Powell\cite{OP}. In the same paper\cite{OP}, they discuss the stable smooth mapping class group of simple connected 4-manifolds under smooth category. However, Smooth mapping class groups remain largely inaccessible due to the lack of practical techniques. In 2025, Lin, Wu, Xie and Zhang\cite{LWXZ} extended Dax invariants from embedded disks to embedded surfaces in $S^2$-bundle and applied the generalization of Dax invariants to show the following surjectivity properties of $MCG(M)$ for $M$ being a trivial $S^2$-bundle over surfaces. Here $\mathbb{Z}^{\infty}$ denote the direct sum of infinitely many copies of $\mathbb{Z}$.

\begin{Theorem}[LWXZ, 2025]
    There exists a surjective homomorphism $\phi: MCG(M) \rightarrow \mathbb{Z}^{\infty}$ such that its restriction on $MCG_0(M)$ is also of infinite rank.
\end{Theorem}

The generalized Dax invariants of embedded surface provides a practical way of analyzing mapping class group of fiber bundles in 4-manifold. This paper is the first-step application of the technique on mapping class group of nontrivial $S^2$ fiber bundles after their work on trivial $S^2$ fiber bundles. We show that the mapping class group of nontrivial $S^2$ fiber bundle whose base space is surface with finite genus $g>1$, whether orientable or not, inherits the surjectivity properties of $MCG(M)$, for all $S^2$ fiber bundle over closed surfaces. Independently, You~\cite{You26} gave an alternative proof of the surjectivity
\[
MCG(\Sigma\ltimes S^2)\twoheadrightarrow \mathbb Z^\infty
\]
for the nontrivial orientable $S^2$-bundle over a closed surface of positive genus, by constructing relative Dax invariants directly in the nontrivial bundle. This suggests that the phenomenon of $Z^{\infty}$ surjectivity is not an artifact of product decomposition, but a property intrinsic to 2-sphere fibrations in the smooth category. In addition, it also shows that the generalization of Dax invariants remains a powerful and innovative tool for studying the algebraic size of families of fiber bundles in 4-manifold.

Consider a nontrivial fiberation $$S^2 \rightarrow X \rightarrow \Sigma$$ and let $X$ be any nontrivial $S^2$ fiber bundle over a closed surface $\Sigma$ with finite positive($g>1$) genus.  %In Section 1.1.1, we give an explicit construction of $\Sigma \ltimes S^2$ and claim that there is only one such orientable $S^2$ fiber bundle over orientable closed surface $\Sigma$. In section 1.1.2, we apply the same arguement to show that there is also only one nontrivial orientable $S^2$ fiber bundle with base space $\Sigma'$ an nonorientable closed surface.

In this paper, we prove the following theorem as our main theorem:

\begin{Theorem}\label{1.1}
    There exists a surjective homomorphism $\Psi': MCG(X) \rightarrow \mathbb{Z}^{\infty}$, whose image of $MCG_0(X)$ is also $\mathbb{Z}^{\infty}$. 
\end{Theorem}

The above theorem also suggests that the Torelli Group\footnote{The \emph{Torelli group}, denoted $\mathcal{I}_g$, is a subgroup of the mapping class group $\mathrm{MCG}(\Sigma_g)$ of a closed oriented surface $\Sigma_g$($g \geq 2$). It is defined as the kernel of the natural homomorphism
\[
\Phi: \mathrm{MCG}(\Sigma_g) \longrightarrow \mathrm{Sp}(2g, \mathbb{Z}),
\]
which sends a mapping class to its induced action on the first homology group $H_1(\Sigma_g; \mathbb{Z})$, preserving the symplectic intersection pairing. That is,
\[
\mathcal{I}_g := \ker\left( \mathrm{MCG}(\Sigma_g) \to \mathrm{Aut}(H_1(\Sigma_g; \mathbb{Z})) \right) = \ker(\Phi).
\]} of $X$ containing $MCG_0(X)$ may be of infinite rank. Before section 3, all bundles are assumed to be orientable.

Section 1.1 provides explicit constructions of the nontrivial $S^2$
-fiber bundles over both orientable and nonorientable closed surfaces and shows the uniqueness of such bundles in each case. We also give a brief introduction to barbell diffeomorphism in section 1.2. In section 1.3, We discuss some properties of Mapping class group of $\Sigma \times S^2$ and application of generalization of Dax invariant in \cite{LWXZ}.  In section 2, we apply the generalization of relative Dax invariant to prove our main theorem \ref{1.1} in the sense of lifting maps in $M$.

\subsection{Nontrivial $S^2$ Fiber bundles over surface}

In this section, we show the construction of two families of orientable nontrivial fiber bundles from the view of handle decompositions, $\Sigma_g \ltimes S^2$ and $\Sigma'_g \ltimes S^2$, where $\Sigma_g$ is an orientable closed surface with genus $g$ and $\Sigma'_g$ is an unorientbale closed surface with genus $g$ for $g$ being an arbitrary positive integer.

For $g=0$, one may view the only nontrivial $S^2$ fiber bundle as $\mathbb{CP}^2\#\mathbb{CP}^2$. In this case, $\pi_1(X) \simeq \pi_1(S^2) = 0$ and there is certainly no surjectivity map to $\mathbb{Z}^{\infty}$.

For $g=1$, $\pi_1(X) \simeq \pi_1(T^2) = \mathbb{Z}^2$. Therefore, it still has infinitely many nontrivial conjugacy classes and so does all positive genus.

\subsubsection{}Let $\Sigma_g$ be a closed orientable surface. Orientable $S^2$-bundles over
$\Sigma_g$ may be realized as unit sphere bundles $S(E)$ of oriented
rank-three real vector bundles $E\to \Sigma_g$. Since oriented rank
$k\geq 3$ vector bundles over a closed surface are classified by their
second Stiefel--Whitney class, $E$ is trivial if and only if
$w_2(E)=0$. Hence, as
\[
H^2(\Sigma_g;\mathbb Z/2)\cong \mathbb Z/2,
\]
there are exactly two orientable $S^2$-bundles over $\Sigma_g$: the
trivial bundle $\Sigma_g\times S^2$, corresponding to $w_2(E)=0$, and a
unique nontrivial bundle, corresponding to the nonzero class
$w_2(E)\neq 0$. One may also use another geometric method called clutching construction to verify by considering $S^2$ as a union of two disks and maps in $Diff^+(D^2) \simeq SO(3)$.

%\textcolor{blue}{This part is an attempt to show that there exists a global section $h: \Sigma \rightarrow X$. By above remark, there exists $f_1$, $f_2$ s.t. $\pi\circ f_1|_{H_0\cup H_1} \simeq id_{H_0\cup H_1}$ and $\pi\circ f_2|{H_2} \simeq id_{H_2}$. Hence one needs to show that $f_1|_{S^1} = f_2|_{S^1}$ and then we can construct a map $h$ s.t. $h|_{H_1\cup H_1} = f_1$ and $h|_{H_2} = f_2$. Thus by the above construction, $\forall x\in S^1$, we have $$\pi^{-1}(x) = S^2 \cong S^2 = \pi^{-1}(x)$$.}

%\textcolor{red}{ Since there is a unique nontrivial way to construct an orientable nontrivial $S^2$ fiber bundle by gluing $S^1 \times S^2$, we may have $f_1(x) = f_2(x) $}. \textcolor{blue}{There is a Gap here}

\subsubsection{}Now we describe the orientbale bundle with fiber $S^2$ over base space that is nonorientable closed surfaces $\Sigma'$ with some finite number $g\geq 2$ i.e. $\Sigma' = \#_{g} \mathbb{RP}^2$. We decompose  $\Sigma'$ into handles s.t. it has 1 0-handle, 1 2-handle and some 1-handles. The orientability of a surface is determined by the framing of the attaching spheres of all 1-handles \cite{GS}. There is a canonical bijection between the set of framings and $\pi_{k-1}(O(n-k))$ for k-handle attachment on $\partial(X^n)$. Attaching 1-handles on 0-handle, we have  $\pi_0(O(1)) \cong \mathbb{Z}_2$, there are only two choices of framings for 1-handle. If all the 1-handles have trivial framing on the attaching sphere, then the surface is therefore orientbale. Thus for an nonorientable surface, there exists at least one 1-handle having nontrivial framing. We start by recalling the handle construction of orientbale closed surface with genus $g$, consisting of 1 0-handle, 2g 1-handle and 1 2-handle. Assigning 1-handles order from 1 to $2g$, we change the framing of some of them. %By Euler characteristic, if $g$ is even, we attach $g$ 1-handles. However, if $g$ is odd, then we need to add one more twisted 1-handle on the union of $g-1$ 1-handles and $H_0$. Fix a handle decomposition of $\Sigma'$ s.t it has 1 0-hande, g 1-handles and 1 2-handle.
Therefore there are $2^{2g}-1$ ways to construct an nonorientable surface from a the handle decomposition of orientable closed surface.  We claim that the above construction gives $2^{2g-1}-1$ different nontrivial fiber bundles over some $\Sigma'$. 

\begin{remark}
    Each $H_1$ can be viewed as a thickened 1-cell, hence $H_0\cup H_1 \simeq \vee_{i\in [1,g]} S^1_i $. The framing of 1-handle doesn't change the homotopy type of $H_0\cup H_1$. Therefore, for fixed $H_0\cup H_1$, there is still only one trivial $S^2$ fiber bundle over $H_0\cup H_1$.
\end{remark}

Suppose there is a bundle isomorphism between any two such fiber bundles $X_1$ and $X_2$, then we may simply restrict the base space to $H_0 \cup H_1$ inheriting the notation from above. Since $X_1$ and $X_2$ are constructed from distinct manipulations of 1-handles attachment of $\Sigma$, the corresponding $H_0\cup H_1$ are therefore different. Then there exists $H^i_1$ in $X_1$ and $X_2$ s.t. the attaching maps have distinct framings. Considering the following fiber bundle $$S^2 \rightarrow X_j|_{H_0\cup H^i_1} \rightarrow H_0\cup H^i_1, \ j=1,2$$. However, the two pullbacks of $H_0\cup H^i_1$ should be isomorphic under the restriction of isomorphism, which is a contradiction.

The $H_0\cup H_1$ is a punctured nonorientable surface with boundary. Observe that while attaching $(H_0\cup H_1) \times S^2$ and $H_2 \times S^2$, the attaching sphere $S^1$ is corresponding to the boundary of $\partial H_2 = \partial D^2$. In $H_0 \cup H_1$, the $S^1$ is homotopic to a loop going through all generator twice. For those mobius band 1-handles, the nonorientation part is vanished. Hence we are back to orientable $H_0\cup H_1$ case as above. By the same argument of clutching construction, there are at most two fiber bundles. Therefore by above subsection, there is only one nontrivial orientable $S^2$ fiber bundle over $\tilde{\Sigma}$.

\subsection{Relative Dax invariant in $\Sigma \times S^2$ and barbell diffeomorphism}

One of the main techniques used in the proof of the main theorem is the self-referential barbell diffeomorphism\cite{BG21} and generalization of relative Dax invariant\cite{LWXZ} for closed embedded surfaces.

\subsubsection{Generalization  of Dax invariant}

The  generalization of relative Dax invariant\cite{LWXZ} is an extension from the Dax invariant constructed by Gabai \cite{BG21} in 2021. Gabai defined the original relative Dax invariant for two disk embeddings in aribtrary 4-manifold. The two embedded disks are restricted to be properly embedded and homotopic to each other relative to the boundary. Considering the 1-parameter family of arcs for each embedded disk, observing that all such arcs in the two families sharing the same endpoints, we can obtain a loop in $Emb_{\partial}(I,M)$ i.e the embedding space of arc in $M$. Fix a handle decompostion of $\Sigma_g$ s.t there is exactly one 0-handle $H_0$, 2g 1-handles $\{H^i_1\}_{i \in \{1, \dots, 2g\}}$ and one 2-handles $H_2$ where $b_0\in H_0$. We denote $H_1$ the union of all 1-handles. Let $U \subset S^2$ be a disk containing $b_1$, denote $M_1 = \overline{M-H_0\times U}$ and $M_2 = \overline{M-(H_0\cup H_1)\times U}$. 

After smoothing corner, both $M_1$ and $M_2$ are codimension 0 submanifold. When we consider embedding space of 0-codim submanifold, the base point is set to be the restriction of the basepoint in the embedding space from the original manifold. The relative Dax invariant is therefore defined based on the Dax isomorphism\footnote{The isomorphism is the same in \cite{LWXZ}: Dax: $\pi_1^D(Emb_{\partial}(I,X_2);I_0) \rightarrow \mathbb{Z}[\pi \backslash \{1\}] / Im(d_3)$. Notice that by chasing diagram $\pi_1(X_2) \cong \pi_1(X)$, we denote it by $\pi$. $\forall [\gamma] \in \pi_1^D(Emb_{\partial}(I,X_2);I_0)$, we send it to $\pi_2(X) = \pi_1(Map_{\partial}(I, X))$. By Smale's Theorem, $\pi_1(Imm(I,X) \cong \pi_1(Map(I,M))$. We can therefore extend the $\gamma$ to a map $h: D^2 \rightarrow Imm_{\partial}(I,X_2)$. After obtaning elements in $\pi$ from double points resolution, the image of $Dax([\gamma])$ is defined to be the sum of these elements in $\pi$ with signs corresponding to local orientations of the double points.} by considering the exact sequence: 

\begin{equation}
    \dots \rightarrow \pi_3(M_2) \xrightarrow{d_3} \mathbb{Z}[\pi_1(M_2)\backslash \{1\}] \xrightarrow{\tau} \pi_1(Emb(I, M_2; I_0) \xrightarrow{F} \pi_2(M_2) \dots
\end{equation}

In 2025, Lin, Wu, Xie and Zhang\cite{LWXZ} extend the Dax invariant from Disk embedding to embedding of closed orientable embedded surface in $\Sigma_g \times S^2$ while $\Sigma_g$ denotes the closed surface with genus $g$. The proof is based on the study of embedding of Handles\footnote{In \cite{LWXZ}, they show that the relative Dax invariant is independent of handle decomposition of $\Sigma$ } of $\Sigma_g$ in the total space. The embedding of $2-$handle of $\Sigma_g$ can be viewed as an embedding of $D^2$ after some manipulations, which fits in the original definition of Gabai's consturction of disk embedding. The relative Dax invariant is extended from Dax isomorphism in \cite{BG21}. They analyze the impact of $1-$handles movement by exploring the following filtration tower.

\begin{equation}
    Emb^{f_0}(\Sigma, M) \xrightarrow{r_2} Emb^{f_0}(H_0 \cup H_1, M) \xrightarrow{r_1} Emb^{f_0}(H_0, M) \xrightarrow{r_0} *
\end{equation}

Where the map $r_1$, $r_2$ are given by restrictions. We may consider $F_i$ to be the preimage of the base point under $r_j$ $\forall i = 0, 1, 2$ and $F_j$ serves as a fiber of $r_j$.

We may write two long exact sequences based on the fibration tower:

\begin{equation}
\pi_2(\mathrm{Emb}^{[f_0]}(H_0, M)) \to 
\pi_1(F_1) \xrightarrow{r_1*} 
\pi_1(\mathrm{Emb}^{[f_0]}(H_0 \cup H_1, M)) \to 
\pi_1(\mathrm{Emb}^{[f_0]}(H_0, M))
\end{equation}

\begin{equation}
\pi_1(\mathrm{Emb}^{[f_0]}(H_0 \cup H_1, M)) \to 
\pi_0(F_2) \xrightarrow{r_2*} 
\pi_0(\mathrm{Emb}^{[f_0]}(\Sigma, M)) \to 
\pi_0(\mathrm{Emb}^{[f_0]}(H_0 \cup H_1, M)) = 0
\end{equation} 
Since $F_0$ is simple connected, $r_1*$ is surjective. Since $\mathrm{Emb}(H_0 \cup H_1, X)$ is also connected due to the fact of $F_1$ being connected, we have:

\begin{equation}\label{2.4}
    q_0: \pi_0(\mathrm{Emb}(\Sigma^{[f_0]}, X)) \rightarrow \pi_0(F_2)/\pi_1(F_1)
\end{equation}

Then they show that the relative Dax map between two $\Sigma_g$ embeddings satisfies all properties of the original Dax invariant. For notation convinience, we use $\Sigma$ as a closed surface with finite genus. Conisder $\{b_0\} \in \Sigma$ be a base point. Here we denote $\mathscr{I}_0: \Sigma \rightarrow \Sigma \times S^2$ be the embedding having $0$ double intersection point with $\{b_0\} \times S^2$ after resolution.

\begin{Theorem}[\cite{LWXZ}]\label{thm: relative Dax invariant}
 Let $\zeta = (\pi_1(M,b)\setminus \{1\})/\text{conjugaton}.$ There exists a map 
 \[
 \Dax: \mathcal{E}_0\times \mathcal{E}_0\to \mathbb{Z}[\zeta]
 \]
 such that the following properties hold.
  \begin{enumerate}
    \item If $i_1$ is smoothly isotopic to $i_2$ relative to $b_0$, then $\Dax(i_1, i_2)=0$.
    \item $\Dax(i_1, i_2)+\Dax(i_2, i_3)=\Dax(i_1, i_3)$.
    \item If $f$ is a diffeomorphism on $M$ that is homotopic to the identity relative to the base point $b$, then $\Dax(i_1,i_2)=\Dax(f\circ i_1,f\circ i_2)$.
     \item For every $g\neq 1\in \zeta$, there exists a diffeomorphism $f:M\to M$ that is homotopic to the identity relative to $b$, such that $\Dax(\mathscr{I}_0, f\circ \mathscr{I}_0)=g+g^{-1}$, and that $\{b_0\}\times S^2$ is a common geometric dual of $\mathscr{I}_0$ and $f\circ \mathscr{I}_0$. 
  \end{enumerate}  
\end{Theorem}

\subsubsection{Self-referential Barbell Diffeomorphism}

We then review the definition of Barbell diffeomorphism by inheriting notations from \cite{BG21}. The model barbell is constructed to be a boundary connected sum of two $S^2\times D^2$, denoted as $N\beta = S^2\times D^2 \# S^2\times D^2$. If we contract along one direction, we obtain $S^2 \times D^1 \# S^2\times D^1$, which looks like a barbell.  

\includegraphics[scale=0.5]{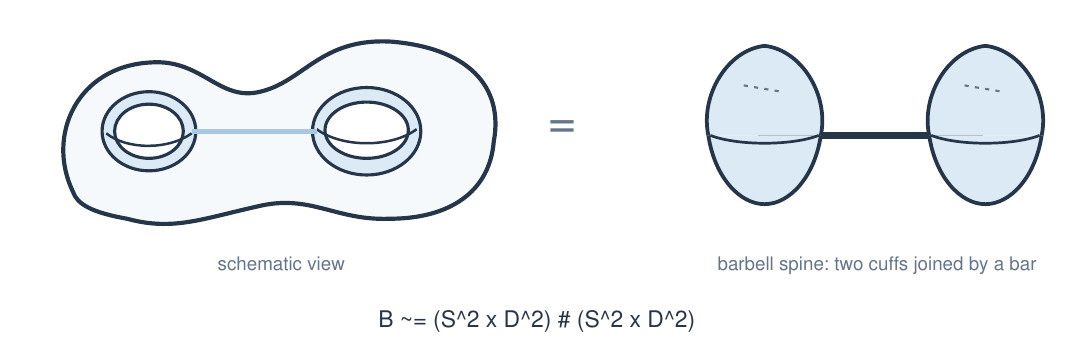}

Suppose we have two disjoint embedded arcs in $D^4$, denoted $I_0$, $I_1$. Let $v(I_0), v(I_1)$ denote neiborhood in $D^4$. Take a time map in $I_1$ and we observe the slice at $t=0$ in $D^3 \times \{0\}$. In this slice, for the $pt \in I_1$, the neighborhood of it is $D^3$. 
\begin{center}
\includegraphics[scale=0.3]{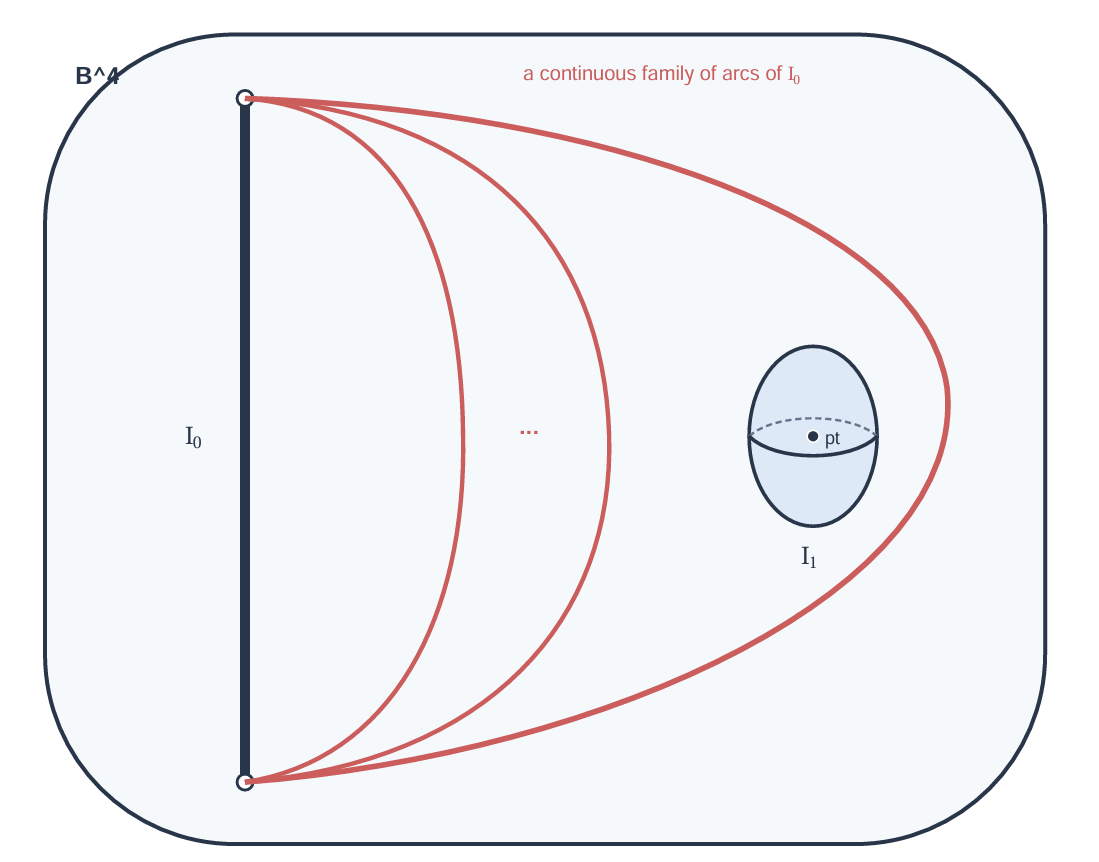}
\end{center}

let $B = D^4\setminus v(I_0\cup I_1) \cong N\beta$. We defines an isotopy $\tilde{\beta}: [0,1]\times [0,1] \rightarrow B$.This extends to a barbell diffeomorphism $\beta: B\rightarrow B$.

That is to say we may reformulate with a well-defined arc pushing map $$\partial: \pi_1(Emb^{Fr}_{\partial}(I,D^4\setminus v(I_1)),I_0) \rightarrow \pi_0(Diff_{\partial}(D^4\setminus v(I_1),I_0)$$. The map is obtained by isotopy extension theorem, i.e any embedding of arcs can be extended to a diffeomorphism in the complement with boundary homotopic. The standard barbell diffeomorphism is a representative $\beta: B \rightarrow B$ whose extension by taking identity on $v(I_0)$ is in $\pi_0(Diff_{\partial}(D^4\setminus v(I_1),I_0)$ as an image over the arc pushing map. Now let $\alpha: N\beta \hookrightarrow X$, we define the diffeomorphism $\beta_\alpha$ of $X$ derived from barbell diffeomorphism $\beta$ s.t. $\beta_\alpha(x) = x$ $\forall x \notin Im(\alpha)$ and $\beta_\alpha(x) = \alpha\circ \beta \circ \alpha^{-1}(x)$ $ \forall x \in Im\alpha$. For more details and properties of barbell diffeomorphisms, see \cite{BG21}.

%Now we fix a basepoint $(b_0,b_1) \in M_2$ and fix an orientation of $\{b_0\} \times S^2 \simeq S^2$ and $\Sigma \times \{b_1\} \simeq \Sigma$. Let $p,q$ be two distinct points in $M_2$, consider an arc $\gamma$ in $\Sigma$ whose interior intersect $q$ transversely but disjoint from $p$. 

%Let $b=(b_0,b_1) \in \Sigma \ltimes S^2 = X$ be a base point. Let $f_0: \Sigma \rightarrow X$ be an embedding s.t. $f_0(x) =(x,b_1)$, denote $im(f_0) = \Sigma_0$. Now notice that there may exist infinitely many intersection points between $\Sigma_0$ and $S^2$. We denote $\Sigma_d$ obtained from $\Sigma_0$ by resolving $d$ intersection points. Let $E_d = \{ i: (\Sigma,b_0 \rightarrow (X,b) | \ i \simeq f_d  \ \textit{rel} \  b_0\}$ be the set of smooth embeddings. In the paper, we will only work with $f_0$.

The class of self-referential barbell diffeomorphism\cite{LWXZ} is defined to be a barbell diffeomorphism on $M_2$. The self referential barbell diffeomorphism is first constructed to explore the property of geometrically embedded sphere dual surfaces in $M$ that are homotopic but not isotopic to each other.

Recall that $H_2$ is the $2$-handle of $\Sigma$ and let $p\neq q$ be interior points of $H_2$.
Take a smooth, not necessarily embedded arc $\gamma$ from $p$ to $q$ whose interior
avoids $p$ and meets $q$ exactly once, transversely. Lift $\gamma$ to an arc
$\hat\gamma\subset M_2$ whose only self-intersection is its endpoint over $q$, and let
$m$ be a meridian of $\hat\gamma$ at that point, oriented so that its linking number with
$\hat\gamma$ is positive. After modifying $\hat\gamma$ near this endpoint, it becomes an
embedded arc from $\{p\}\times S^2$ to $m$. The \emph{self-referential barbell
diffeomorphism} $\beta_{p,\gamma}$ is the barbell diffeomorphism of
$(\{p\}\times S^2,\,m,\,\hat\gamma)$, with $\{p\}\times S^2$ oriented by the $S^2$-factor
of $M$, extended by the identity over $M_2$.

\begin{remark}
$\beta_{p,\gamma}$ is determined by $(p,\gamma)$ up to isotopy, so we regard it as an
element of $\pi_0\mathrm{Diff}_\partial(M_2)$. Any two such diffeomorphisms can be
isotoped to have disjoint support, hence they commute, and we write their composition and
quotient additively.
\end{remark}

Let $\mathfrak{p}:\mathbb Z[\pi\setminus\{1\}]\to\mathbb Z[\zeta]$ be the
canonical projection sending each generator to its conjugacy class. For
$\gamma$ as above, let $c(\gamma)$ be the loop based at $q$ that follows
$\gamma$, oriented as $\gamma$; it defines a conjugacy class in $\pi$. Write
$[c(\gamma)]\in\mathbb Z[\zeta]$ for this class if $c(\gamma)$ is nontrivial,
and $[c(\gamma)]=0$ otherwise. We call $c(\gamma)$ the associated loop with the
self-referential barbell diffeomorphism $\beta_{p,\gamma}$ and we will use the following lemma in \cite{LWXZ}:

\begin{lemma}\label{lem:self-ref}
Let $\beta_{p,\gamma}$ be a self-referential barbell diffeomorphism. Then
\begin{enumerate}
  \item $\beta_{p,\gamma}$ is pseudo-isotopic to the identity relative to $\partial M_2$;
  \item we have
  \[
    \mathfrak{p}\circ\mathrm{Dax}\big(\beta_{p,\gamma}(D_0),D_0\big)
    =\epsilon\big([c(\gamma)]+[c(\gamma)^{-1}]\big),
  \]
  where $\epsilon\in\{1,-1\}$ is a universal constant depending only on the
  orientation conventions.
\end{enumerate}
\end{lemma}

\subsection{Mapping class group of $\Sigma \times S^2$}

In fact, one can form a homomorphism by relative Dax map $\phi_1: MCG_0(M) \rightarrow \mathbb{Z}[\zeta]$ sending $[f] $to $Dax(\mathscr{I}_0,f\circ \mathscr{I}_0)$. \cite{LWXZ} shows that for trivial fiber bundle $M = \Sigma \times S^2$, there is always a surjective homomorphism $\phi: MCG_0(M) \rightarrow \mathbb{Z}^{\infty}$. $\phi$ is derived from $$\phi_0: MCG_0(M) \rightarrow \mathbb{Z}[\zeta]$$ sending $[f]$ to $Dax(f_0,f\circ f_0)$. Here the Dax map refers to the relative Dax invariant. Let $\phi_0$ be such a homomorphism satisfying relative Dax properties of (2) and (3) in \cite{LWXZ}.  %By property 4), the image of $\phi_0$ is of infinite rank. Therefore $Im(\phi_0) \cong \mathbb{Z}^{\infty}$ as the subgroup of $\mathbb{Z}[\zeta]$.
Futhermore, let $A$ denote the set of orbits under group action $Aut(\pi) \times \mathbb{Z}/2$ on $\zeta$, here $\mathbb{Z}/2$ is an inverse action. Considerng the quotient map $q: \mathbb{Z}[\zeta] \rightarrow \mathbb{Z}[A]$. The composite map
\[
\begin{tikzcd}[column sep=large, row sep=large]
MCG(M) \arrow[r,"Dax"] &
\mathbb{Z}[\zeta] \arrow[r,"q"] &
\mathbb{Z}[A]
\end{tikzcd}
\]
is therefore the infinity surjective map in the following theorem.

\begin{Theorem}
    There exists a surjective homomorphism $\phi: MCG(M) \rightarrow \mathbb{Z}^{\infty}$ such that its restriction on $MCG_0(M)$ is also of infinite rank.
\end{Theorem}

Notice that for any embedded arc in $M$, there is a barbell diffeomorphism $\beta$ that is homotopic to identity\cite{BG21}. %Hence $\beta \in MCG_0(M)$ and by \cite{LWXZ}[lemma 4.16] there exists a surjective homomorphism from $\{\beta\}$ to $\mathbb{Z}[\zeta]^{\sigma} \simeq <[g]+[g]^{-1}>$ under conjugacy as a subgroup of $\mathbb{Z}[\zeta]$. Hence we have a surjective homomorophism from $MCG_0(M)$ to $\mathbb{Z}[\zeta]^{\sigma} \simeq \mathbb{Z}^{\infty}$.

\section{On the Mapping Class Group of $\Sigma \ltimes S^2$}

In this section, we prove the main theorem \ref{1.1} by exploring the double cover of $X$ using relative Dax invariant of $\Sigma\times S^2$ constructed in \cite{LWXZ}.

Recall that $X = \Sigma \ltimes S^2$. We have constructed $X$ using handle decomposition of $\Sigma$ and the clutching construction in Section 1.1. Let $X^{(2)}$ be the double cover of $X$ with covering map $p: X^{(2)} \rightarrow X$, we therefore study the lift of diffeomorphisms of $MCG_0(X)$ in $MCG_0(X^{(2)})$ using the relative Dax invariant. Canonically, the lift in $MCG_0(X^{(2)})$ is a lift of $\phi\circ p$ for $p$ being the covering map and $\phi \in MCG_0(X)$. Notice that all cyclic covers of a surface is still a surface and $X^{(2)}$ is still a 4-dimensional Manifold. In the meanwhile, if we consider $S^2$ bundle over non-orientable surfaces $\Sigma'$, we can lift $\Sigma'$ to orientable double cover and then lift $\Sigma'\ltimes S^2$ to orientable double cover. The approach is regardless of whether such fiber bundle is orientable or not. Nonorientable bundle case will be discussed in section 3.

Suppose $\Sigma$ has genus g, let $p: \Sigma^{(2)} \rightarrow \Sigma$ be the covering map with the double cover $\Sigma^{(2)}$ of $\Sigma$ where $\Sigma^{2}$ is also a closed surface with genus $2g-1$ by Euler Characristic. We consider the pullback fiber bundle $p^*X$ and claim that $p^*X$ is therefore the double cover of $X$. $$p^*X = \{(a,x)| a \in  \Sigma^2, x \in X, \pi(x) = p(a) \}$$. $\forall a \in \Sigma^{(2)}$, $\pi'^{-1}(a) = (a,x)$ s.t. $\pi(x) = p(a)$ $\implies$ $x \in \pi^{-1}( p(a)) \cong S^2$. Following the definition of fiber bundle $S^2 \rightarrow X \xrightarrow{\pi} \Sigma$, we have $\pi^{-1}\circ p(a) \cong p(a) \times S^2 \cong S^2$.

Let $\pi'$ denote the natural projection s.t. $\pi'(a,x) =a$, we therefore have the following pullback fiber bundle:
$$S^2 \rightarrow p^*X \xrightarrow{\pi'} \Sigma^{(2)}$$

\begin{Lemma}\label{double-trivial}
    For $X = \Sigma \ltimes S^2$, the double cover of $X$ is $X^{(2)} \cong p^*X \cong \Sigma^{(2)} \times S^2$.
\end{Lemma}

\begin{proof}
\begin{enumerate}

\item We first prove that $X^{(2)} \cong p^*X$. Let
\[
\tilde{p}: p^*X \to X, \quad (a, x) \mapsto x.
\]
be the natural projection map
We claim that $\tilde{p}$ is a $2$-sheeted covering map.

Since $p^*X$ is a $S^2$ fiber bundle, observe that for each point $x \in X$, there exists a neighborhood $U \subset \Sigma$ of $\pi(x)$ with a local trivialization in $\Sigma \times S^2$
\[
\varphi: \pi^{-1}(U) \to U \times S^2.
\]
Since $p: \Sigma^{(2)} \to \Sigma$ is a covering map, the preimage $p^{-1}(U)$ is a disjoint union of two open sets $U_1, U_2 \subset \Sigma'$ each homeomorphic to $U$ via $p$. Over each $U_i$, the pullback bundle satisfies:
\[
(p^*X)|_{\pi'^{-1}(U_i)} = \{ (a, x) \in U_i \times \pi^{-1}(U) \mid p(a) = \pi(x) \} \cong U_i \times S^2.
\]
Hence, over each $U$, we obtain:
\[
\tilde{p}^{-1}(\pi^{-1}(U)) = (p^*X)|_{p^{-1}(U)} = (U_1 \times S^2) \sqcup (U_2 \times S^2),
\]
and the map $\tilde{p}$ on this union simply collapses each $(a, f) \in U_i \times F$ to $(p(a), f) \in U \times S^2 \cong \pi^{-1}(U)$. Therefore, each point $x \in X$ has a neighborhood evenly covered by $\tilde{p}$, with exactly two disjoint preimages. It follows that $\tilde{p}: p^*X \to X$ is a $2$-sheeted covering map.

\item We now prove that $X^{(2)} \cong \Sigma^2\times S^2$. Fix $\Sigma$ a handle decomposition with one $0$–handle, $2g$ $1$–handles,
and one $2$–handle attached along the loop
$\beta:\partial H_2\cong S^1 \longrightarrow H_0\cup H_1$,
representing the relator $\prod_{i=1}^g[a_i,b_i]\in\pi_1(H_0\cup H_1)$.
Any orientable $S^2$–bundle over $\Sigma$ is obtained by clutching two trivial
bundles over $H_0\cup H_1$ and over $H_2$ respectively along the overlap
$\partial H_2\simeq S^1$ by a map$\phi:S^1\longrightarrow \mathrm{SO}(3)$.
Isomorphism classes are determined by the homotopy class 
$[\phi]\in \pi_1(\mathrm{SO}(3))\cong \mathbb{Z}/2$ and the nontrivial bundle
corresponds to $[\phi]=1$.

Now pull back along the connected double cover $p:\Sigma^{(2)}\to \Sigma$. The overlap circle
$\partial H^2\simeq S^1$ lifts to the disjoint union of two circles
$S^1\sqcup S^1$ in $\Sigma'$, and the clutching along the cover is given by the
disjoint union of the two lifts of $\phi$:
\[
\tilde\phi \;=\; \phi\ \sqcup\ \phi \;:\; (S^1\sqcup S^1)\longrightarrow \mathrm{SO}(3),
\]
one on each lifted boundary component.

Since $\pi_1(\mathrm{SO}(3))\cong \mathbb{Z}/2$, the total clutching
class of the pullback bundle is the \emph{sum} of the two boundary classes:
$[\tilde\phi] \;=\; [\phi] + [\phi] \;=\; 1+1 \;=\; 0 \in \mathbb{Z}/2$.
Therefore the pullback bundle is classified by the trivial element and hence is
isomorphic to the trivial bundle $\Sigma^{(2)}\times S^2$.
\end{enumerate}
\end{proof}

\begin{remark}

The proof also works for $\Sigma'$ since the double cover of $\Sigma'$ is also an orientable surface with finite genus.
\end{remark}

\begin{Corollary}
    For $X' = \Sigma' \ltimes S^2$, the double cover of $X$ is $X'^{(2)} \cong p^*X' \cong \Sigma'^{(2)} \times S^2$.
\end{Corollary}

Let $MCG(X)$ be the mapping class group of $X$, i.e. it is a group of diffeomorphisms of $X$ up to smooth isotopy. Let $MCG_0(X)$ denote the group of diffeomorphisms of $X$ i.e. each diffeomorphism is homotopic to identity. 

\begin{Theorem}[\cite{LWXZ}, 2025]\label{1}
    There is a surjective homomorphism $\phi: MCG_*(\Sigma \times S^2) \rightarrow \mathbb{Z}[\zeta]^{\sigma}$, where the restriction on $MCG_0(\Sigma \times S^2)$ is also a surjective homomorphism.
\end{Theorem}

We will use the above theorem to prove Theorem \ref{1.1}. To show that, we now lift maps of $MCG_0(X)$ to $Diff(X^{(2)})$. Let $\varphi \in MCG_0(X)$ be a diffeomorphism of $X$ homotopic to the identity. Since $(\varphi \circ p)_* = \varphi_*\circ p_* \simeq id_*\circ p_*$, we have $ \varphi_*\circ p_*(X^{(2)}) \subset p_*(X^{(2)})$. Therefore, there exists a lift of $\varphi \circ p : X^{(2)} \rightarrow X$, we denote such a lift $\widetilde{\varphi}$. We therefore have the following diagram commutes.

\[
\begin{tikzcd}
X^{(2)} \arrow[r, dashed, "\widetilde{\varphi}"] \arrow[d, "p"'] & X^{(2)}\arrow[d, "p"] \\
X \arrow[r, "\varphi"'] & X
\end{tikzcd}
\]

We claim that $\widetilde{\varphi}: X^{(2)} \rightarrow X^{(2)}$ is a diffeomorphism. Notice that since \( p \) and \( \varphi \) are smooth, since \( p \circ \widetilde{\varphi} = \varphi \circ p \), it follows that \( \widetilde{\varphi} \) is smooth as well. To show that \( \widetilde{\varphi} \) is a diffeomorphism, we construct its inverse map as follow. Since \( \varphi \) is a diffeomorphism, \( \varphi^{-1}: X \to X \) is smooth as its inverse map. Then the composition \( \varphi^{-1} \circ p: X^{(2)} \to X \) is a smooth map from \( X^{(2)} \) to \( X \). Because \( p \) is a covering map, the composition \( \varphi^{-1} \circ p \) lifts uniquely to a smooth map \( \widetilde{\varphi}^{-1}: X^{(2)} \to X^{(2)} \) satisfying
\[
p \circ \widetilde{\varphi}^{-1} = \varphi^{-1} \circ p.
\]

Now we verify that \( \widetilde{\varphi}^{-1} \) is indeed the inverse of \( \widetilde{\varphi} \). For all \( \widetilde{x} \in X^{(2)} \), we compute:
\[
p \circ (\widetilde{\varphi}^{-1} \circ \widetilde{\varphi})(\widetilde{x}) = \varphi^{-1} \circ p \circ \widetilde{\varphi}(\widetilde{x}) = \varphi^{-1} \circ \varphi \circ p(\widetilde{x}) = p(\widetilde{x}).
\]
Thus, \( \widetilde{\varphi}^{-1} \circ \widetilde{\varphi} \) and \( \mathrm{id}_{\widetilde{X}} \) are both lifts of the map \( p \) and agree at each point. We conclude that
\[
\widetilde{\varphi}^{-1} \circ \widetilde{\varphi} = \mathrm{id}_{{X^{(2)}}}.
\]
A symmetric argument shows that \( \widetilde{\varphi} \circ \widetilde{\varphi}^{-1} = \mathrm{id}_{X^{(2)}} \), hence \( \widetilde{\varphi} \) is a diffeomorphism between two $X^{(2)}$. 
Using analogous argument above, it is easy to see that $\widetilde{\varphi}' \in MCG_0(X'^{(2)})$ as a lift of $p'\circ \phi'$ and the following diagram commutes:

\[
\begin{tikzcd}
X'^{(2)} \arrow[r, dashed, "\widetilde{\varphi}'"] \arrow[d, "p'"] & X'^{(2)}\arrow[d, "p'"] \\
X' \arrow[r, "\varphi'"] & X'
\end{tikzcd}
\]

\begin{Theorem}\label{two-lift}
    For each $\varphi$ $\in MCG_0(X)$, there are exactly two lifts of $\varphi$ between $X^{(2)}$, denoted $\varphi_1$, $\varphi_2$ s.t. $\varphi_1 = \gamma \varphi_2$ where $\gamma$ corresponds to deck transformation of $p$.
\end{Theorem}

\begin{proof}
Suppose \( \widetilde{\varphi}_1 \) and \( \widetilde{\varphi}_2 \) are two lifts of \( \phi \circ p \), i.e., they both satisfy
\[
p \circ \widetilde{\varphi}_1 = \varphi \circ p = p \circ \widetilde{\varphi}_2.
\]
Then for all \( a \in X^{(2)} \), we have \( p(\widetilde{\varphi}_1(a)) = p(\widetilde{\varphi}_2(a)) \), so both \( \widetilde{\varphi}_1(a) \) and \( \widetilde{\varphi}_2(a) \) lie in the same fiber of \( p \). Since \( p \) is a 2-sheets covering, each fiber contains exactly two points. Thus, for each \( a \in X^{(2)} \), either
\[
\widetilde{\varphi}_2(a) = \widetilde{\varphi}_1(a) \quad \text{or} \quad \widetilde{\varphi}_2(a) = \gamma(\widetilde{\varphi}_1(a)),
\]
where \( \gamma \) is the nontrivial deck transformation of \( p \). By the uniqueness of lifting, we either have $\widetilde{\varphi}_1 =\widetilde{\varphi}_2$ or  \( \widetilde{\varphi}_2 = \gamma \circ \widetilde{\varphi}_1 \).

Additionally we verify that \( \gamma \circ \widetilde{\varphi} \) is also a lift:
\[
p \circ (\gamma \circ \widetilde{\varphi}) = p \circ \widetilde{\varphi} = \phi \circ p,
\]
Therefore, both \( \widetilde{\varphi} \) and \( \gamma \circ \widetilde{\varphi} \) are lifts, and these are the only two.
\end{proof}

Recall that $\zeta = (\pi_1(M,b)\setminus \{1\})/\text{conjugaton}.$ Reccall that $\Dax: \mathcal{E}_d\times \mathcal{E}_d\to \mathbb{Z}[\zeta]$, denote $D(\phi)= Dax(f_0, \phi \circ f_0)$. Denote $L$ as the lifting map, consider the following composition map $$MCG_0(X) \xrightarrow{L} MCG_0(X^{(2)}) \xrightarrow{D} \mathbb{Z}[\zeta]$$. %We want to show that such map is surjective.To have the lifting map well-defined, denoted $L$, we first show that $\forall$ $\varphi \in MCG_0(X)$, there is only one lift as an element in $MCG_0(X^{(2)})$. $\forall \alpha_1\in X^{(2)}$, $\exists$ an even nbhd $U$ s.t. $p(\alpha_1), \varphi\circ p(\alpha_1) \in U$. $p^{-1}(U) = U_1 \sqcup U_2 $, WLOG say $x_1\in U_1$ and $p|_{U_1}$ is a local homeomorphism. let $\widetilde{\varphi}_1|_{U_1} = p^{-1}|_{U_1}\circ\varphi\circ p|_{U_1}$, which gives a homeomorphism on $U_1$. As we have shown above, $\phi_1$ as a lift of $\phi$ is a diffeomorphism between $X^{(2)}$. Since $\phi\in MCG_0(X)$, there is a homotopy $H: X \times [0,1] \rightarrow X   $ s.t. $H_0 = \varphi$, $H_1(X) = id_X, H_t: X \rightarrow X$ is a continuous map for $t\in (0,1)$. Then we have an homotopic between $\phi|_{U_1}$ and $id|_{U_1}$ by compostion $p^{-1}|_{U_1}\circ H \circ p|_{U_1}$ by homotopy lifting property. Then we glue all such $U_1$ for all $\alpha \in X^{(2)}$ and obtain a homotopy between $\widetilde{\varphi_1}$ and  $id|_{X^{(2)}}$. Thus $\widetilde{\varphi_1} \in MCG_0(X^{(2)})$. Since $\widetilde{\varphi_2} = \gamma\circ \widetilde{\varphi_1}$, the nontrival deck transformation map $\gamma$ is not homotopic ot identity, $\widetilde{\varphi_2}$ is therefore not homotopic to identity and $\widetilde{\varphi_2}$ is not in $MCG_0(X^{(2)})$. Thus $L: MCG_0(X) \rightarrow MCG_0(X^{(2)})$ is well defined.

Take $[\varphi]\in MCG_0(X)$. By definition there exists a homotopy
\[
H:X\times I \longrightarrow X, \qquad H_0=\varphi,\quad H_1=\id_X.
\]
Consider the composition
\[
H\circ(p\times \id_I): X^{(2)}\times I \longrightarrow X.
\]
By the homotopy lifting property, there exists a unique lift
\[
\widetilde H: X^{(2)}\times I \longrightarrow X^{(2)}
\]
with terminal condition $\widetilde H_1=\id_{X^{(2)}}$ such that
\[
p\circ \widetilde H = H\circ(p\times \id_I).
\]
Define $\widetilde\varphi = \widetilde H_0$. Then $p\circ \widetilde\varphi
= H_0\circ p = \varphi\circ p$, so $\widetilde\varphi$ is a lift of $\varphi$.
Moreover, the homotopy $\widetilde H$ shows that $\widetilde\varphi\simeq
\id_{X^{(2)}}$. Above shows that $\widetilde{\varphi}$ is a diffeomorphism between $X^{(2)}$. Hence $[\widetilde\varphi]\in MCG_0(X^{(2)})$.

Suppose $\widehat\varphi$ is another lift of $\varphi$. By above theorem,
$\widehat\varphi = \gamma\circ \widetilde\varphi$ for $\gamma$ the degree 2
nontrivial deck transformation. If $\widehat\varphi\simeq \id_{X^{(2)}}$, then
\[
\gamma = \widehat\varphi\circ \widetilde\varphi^{-1} \simeq \id_{X^{(2)}},
\]
contradicting the fact that the nontrivial deck transformation is not
homotopic to the identity. Thus exactly one lift of $\varphi$ lies in
$MCG_0(X^{(2)})$.

Finally, the class $[\widetilde\varphi]$ does not depend on the choice of
homotopy $H$: if $H$ and $H'$ are two homotopies from $\varphi$ to $\id_X$,
their lifts $\widetilde H, \widetilde H'$ both satisfy the same terminal
condition $\id_{X^{(2)}}$, so by uniqueness of homotopy lifting we have
$\widetilde H_0 = \widetilde H'_0$. Thus $L$ is well defined.

\begin{remark}
    $L: MCG_0(X) \rightarrow MCG_0(X^{(2)})$ is not a surjective map. Due to the commutativity, $Im(L)$ contains only diffeomorphisms of $X^{(2)}$ preserving fibers under $p$.
\end{remark}

\begin{lemma}\label{2.7}
    Let $f$ be a self-referential barbell diffeomorphism with $D(f) = \pm([c]\pm[c]^{-1})$, and let $\tilde{f}$ be a lift of it in a finite cyclic cover. $D(\tilde{f}) = \Sigma_{i\in I} \phi_i(\pm[c]\pm[c]^{-1})$ where $\phi_i$ denotes each deck transformation.
\end{lemma}

\begin{proof}
$f$ is defined on two disjoint points $\alpha, \beta \in X_2$ and $[f]\in \pi_0(Diff_{\partial} X_2)$. Therefore, we have local diffeomorphisms from neighborhoods of double point say $\alpha$ to neighborhoods of all its lifted points and $[\tilde{f}]\in \pi_0(Diff_{\partial} \tilde{X_2})$. Since Dax map counts nontrivial homotopy loops from double point resolution with assigned local orientations by Lemma \ref{lem:self-ref}, we have $D(\tilde{f}) = \Sigma_{i\in I} \phi_i(\pm[c]\pm[c]^{-1})$ for any lift $\tilde{f}$.

\end{proof}

\begin{remark}

Notice that the above lemma is independent of the framings of 1-handles attachment in $\Sigma$ and it is independent of orientation of $\Sigma$ and $X$.
\end{remark}

\begin{Theorem}\label{Main0_1}
    There exists surjective homomorphism $\Psi_0: MCG_0(X) \rightarrow \mathbb{Z}^{\infty} $ from the composition map $D\circ L: MCG_0(X) \rightarrow \mathbb{Z}[\zeta]^{\sigma}$.
\end{Theorem}

\begin{proof}
    We show that $Im(D\circ L) = \mathbb{Z}$. Since $Dax$ is a homomorphism, $D\circ L$ is also a homomorphism. We use the idea of self-referential disk.

    $\forall \beta$ as a self-referential barbell diffeomorphism, $\beta$ is homotopic to identity\cite{BG21}. Hence $\beta \in MCG_0(X)$. We then lift $\beta$ to $MCG_0(M)$,  by above lemma the lifted $\widetilde{\beta}$ is a combination of two self-referential barbell diffeomorphisms with associated loop $\pm([c_1]+[c_1]^{-1}) \pm([c_1']+[c_1']^{-1})$ where $[c_1]= \gamma[c_1']$. Since we can isotope two self-referential barbell diffeomorphisms to have disjoint support, any two such diffeomorphisms are commutative\cite{LWXZ}. By
     Lemma \ref{2.7} we have $$D\circ L(\beta) = D(\widetilde{\beta}) = Dax(f_0,\widetilde{\beta}f_0) = \pm([c_1]+[c_1]^{-1}) \pm([c_1']+[c_1']^{-1})$$. By [LWXZ], every pseudo-isotopy class in \(MCG_0(M)\) is represented by a self-referential barbell diffeomorphism. $Im(D\circ L) = < \pm([c_1]+[c_1]^{-1}) \pm([c_1']+[c_1']^{-1}) >_{[c] \in \pi_1(X)}$

    We then show that $Im(D\circ L)$ is a subgroup of $\mathbb{Z}^{\infty}$. $\forall [c]\in \pi_1(M)$, consider $\gamma(\pm([c]+[c]^{-1}))= \pm([c']^{-1}+[c'])$. Since $\gamma$ has degree 2, it is equivalent to consider a group action of $ \mathbb{Z}_2 $ on $\{\pm([c]+[c]^{-1}) \}= \mathbb{Z}[\zeta]^{\sigma} $ and each orbit consists of at most two homotopy classes: $\pm([c]+[c]^{-1})$ and $\pm([c']+[c']^{-1})$ up to conjugacy. Since $\mathbb{Z}[\zeta]^{\sigma}$ is an infinite group with infinite basis, there are infinite orbits as well under $\mathbb{Z}_2$ action. Therefore, each orbit gives different sums and we have infinite many independent elements $\pm([c]+[c]^{-1})$ $\pm([c']+[c']^{-1})$, which can be represented by $\mathbb{Z}^{\infty}$.

    There exists a surjective homomorphism from the set of self-referential barbell diffeomorphisms of $X$ to $Z^{\infty}$. Therefore such $\Psi$ exists.

   % We show that all $\beta$ are lineararly independent in $MCG_0(X)$. If there exists a linear combination of $\beta$ equal to zero
    %Let $\widetilde{\beta'}$ be a lift of $\beta'$ n $MCG_0(\Sigma\times S^2)$ with associated loop $\pm([c_1']+[c_1']^{-1}) \pm([c_2']+[c_2']^{-1})$s.t $\pm([c_1']+[c_1']^{-1}) \pm([c_2']+[c_2']^{-1})$ = $\pm([c_1]+[c_1]^{-1}) \pm([c_2]+[c_2]^{-1})$. Hence, $D\circ L(\beta') = D(\widetilde{\beta'}) = Dax(f_d,\widetilde{\beta'}) = \pm([c_1']+[c_1']^{-1}) \pm([c_2']+[c_2']^{-1})= D\circ L(\beta)$. We need $\beta'$ isotopic to $\beta'$ in $MCG_0(X)$ so that $\{\beta\}  \nsubseteq ker(D\circ L)$. Since $L$ is a 1-1 map, we have $D(\widetilde{\beta}-\widetilde{\beta'}) = D(\widetilde{\beta})-D(\widetilde{\beta'}) = 0$.

    %Thus $Im(D\circ L)$ is a subgroup in $\mathbb{Z}[\zeta]^{\sigma}$ and $Im(D\circ L) \simeq \mathbb{Z}^{\infty}$ as subgroup of $\mathbb{Z}^[\zeta]^{\sigma}$.
\end{proof}

By analogous proof, we have:

\begin{Corollary}\label{C2.10}
    There exists surjective homomorphism $\Psi_1: MCG_0(X') \rightarrow \mathbb{Z}^{\infty} $. 
\end{Corollary}

\begin{proof}
    The double cover of the nonorientable base surface $\Sigma'$ should be chosen uniquely to be $q:\tilde\Sigma\to\Sigma'$ be the orientation double cover, s.t $q_*\pi_1(\tilde\Sigma)=\ker w_1(\Sigma')$. And we use the construction of pullback bundle following above theorems.
\end{proof}
Furthermore, we can extend the map 
\[
\Psi : MCG_0(X) \longrightarrow \mathbb{Z}^{\infty}
\]
to the subgroup of $MCG(X)$:
\[
LMCG(X)=\{\, f\in MCG(X)\mid 
f_{*}\big(p_{*}(\pi_{1}(X^{2}))\big)=p_{*}(\pi_{1}(X^{2})) \,\}.
\]
Since $p\colon X^{2}\to X$ is a double covering, the subgroup
$p_{*}(\pi_{1}(X^{2}))<\pi_{1}(X)$ has index~$2$ and is therefore normal.  
For a mapping class $f\in MCG(X)$, a lift of $f\circ p$ exists if and only if
\[
f_{*}\big(p_{*}(\pi_{1}(X^{2}))\big)=p_{*}(\pi_{1}(X^{2})),
\]
so $LMCG(X)$ is precisely the subgroup of mapping classes that admit lifts in double cover.

Since $\pi_{1}(X^{2})\cong \pi_{1}(\Sigma^{2})$, the only situation in which every mapping class of $X$ lifts to $X^{2}$ is when the fiber~$F$
has genus~$0$.  In higher genus, there exist mapping classes that do not
preserve $p_{*}(\pi_{1}(X^{2}))$, so we restrict to $LMCG(X)$. Remark that
\[
MCG_0(X)\subset LMCG(X),
\]. if $f\in MCG_0(X)$, then $f$ is homotopic to the identity.  
Consequently, $f_{*}$ acts on $\pi_{1}(X)$ by an inner automorphism.  
Because index-$2$ subgroups are normal, every inner automorphism preserves 
$p_{*}(\pi_{1}(X^{2}))$.  

%By Theorem \ref{two-lift}, we know $\forall f \in LMCG(X)$, there are exactly two lifts in $MCG(X^{2})$ where the two lifts are related to each other by nontrivial deck transformation. Thus it is natural to consider lifts of $f$ in $MCG_{*}(X^{2})$. Define
%\[L : LMCG(X) \longrightarrow MCG_{*}(X^{2})
%\]
%by sending a fiber preserving diffeomorphism $f \in LMCG(X)$ to its unique basepoint–preserving lift. The lifting map $L$ is a well defined homomorphism under the restriction. Remark also that $MCG_0(M) \subset MCG_*(M)$.

Since the fiber in the covering map may not be preserved under homotopy, $\Phi$ can not be extended to all $MCG(X^2)$. To have such lifting map, one needs to consider cover space of $X$ with higher sheets. 

\begin{Theorem}\label{Theorem 2.11}
    There exists an even integer $2k$ and a covering space $X^{2k}$ of $X$ s.t. $(f\circ p_{2k})_*(\pi_1(X^{2k})) = p_{2k}*(\pi_1(X^{2k}))$ for any $f \in MCG(X)$, where $p_{2k}: X^{2k} \rightarrow X$ is the covering map.
\end{Theorem}

\begin{proof}
Take $2m$ to be an even integer for arbitrary positive $m$, we define the set $S = \{H_i \ | \ [\pi_1(X) : H_i] = 2m\}$. Let $L = \cap_{H_i \in S} H_i$. Hence $f_*(L) = \cap_{H_i \in S} f_*(H_i) = \cap_{H_j \in S} H_j = L$ since $f \in MCG(X)$ is bijective. 

It is easy to see that $L$ is nontrivial since $[\pi_1(X) : L] = [\pi_1(X): H_i] [ H_i: L] = 2m \times n$ for some $n$. Since $\pi_1(X)$ is an infinite group for $g \neq 0$, $L$ is in fact a nontrivial subgroup of even index, denote $k=mn$. Take $X^{2k}$ to be the corresponding cover space of $L$, the proof is therefore complete.
\end{proof}
Before the application of generalization of Dax invariants in $MCG(X^{2k})$, we provide a straightforward proof of the following lemma:
\begin{lemma}
The $2k$-fold cover of $X$
is diffeomorphic to $S^2 \times \Sigma_g$, for some $g$.
\end{lemma}

\begin{proof}
Since the corresponding subgroup of $H^{2k} < H^2$, there exists $q: X^{2k}\rightarrow X^2$. It is equivalent to show that any fiber preserving cover of $X^2$ as trivial bundle is also a trivial bundle. Since $p^{2k}=p^2\circ q$, the bundle $p^{2k}$ is the pullback of the trivial bundle
$p^2\colon S^2\times\Sigma_g\to\Sigma_g$ and hence is itself trivial.
\end{proof}

We define $\Psi' = q \circ Dax \circ L_{2k}$ where $L_{2k}: MCG(X) \rightarrow MCG(X^{2k})$ is the lifting map.  
\[
\begin{tikzcd}[column sep=large, row sep=large]
MCG(X) \arrow[r,"L_{2k}"] \arrow[rrr,bend left=25,"\Psi'"'] &
MCG(X^{2k}) \arrow[r,"Dax"] &
\mathbb{Z}[\zeta] \arrow[r,"q"] &
\mathbb{Z}[A]
\end{tikzcd}
\]
Notice that $Im(L_{2k})$ refer to one of $2k$ lifted maps $\forall f \in MCG(X)$, we may fix one lifted map. It si independent of choices since $Dax$ take each lifted map $\tilde{f}_i$ to the same image i.e. $Dax(\tilde{f_i}) = \sum_{i=1}^{2k} \phi_i\circ Dax(f)$ where $\phi_i$ denotes a deck transformation. Thus $\Psi'$ is well defined.

We can now extend the Theorem \ref{double-trivial}, Theorem \ref{two-lift} and Theorem \ref{Main0_1} to the $X^{2k}$ cover and obtain the following corollary:

\begin{Lemma}\label{02k}
The image of the restriction map on $MCG_0(X)$ is $\mathbb{Z}^{\infty} $.
\end{Lemma}

\begin{proof}
$\Psi'$ is a homomorphism by Lemma 5.11 in \cite{LWXZ}. Since every self-referential barbell diffeomorphism is in $MCG_0$, the image of all barbell diffeomorphisms under $Dax \circ L_{2k}$ is therefore the subgroup $B$ containing $\sum_{i=1}^{k} \phi_i(\pm([c] + [c]^{-1}))$ for all nontrivial ${[c] \in \ pi_1(X)}$. It is equivalent to consider a group action of $ \mathbb{Z}_{2k} $ on $\{\pm([c]+[c]^{-1}) \}= \mathbb{Z}[\zeta]^{\sigma} $ and each orbit consists of $2k$ homotopy classes: $\pm([c]+[c]^{-1})$ and image of $\pm([c]+[c]^{-1})$ under all deck transformations, up to conjugacy. Since $\mathbb{Z}[\zeta]^{\sigma}$ is an infinite group, there are infinite orbits as well. The rest follows from Theorem \ref{Main0_1} and there are infinitely many orbits under $\mathbb{Z}_{2k}$. We denote this restriction map $\Psi'_0$.

Thus, there are infinitely many $\Sigma_{i=1}^{2k} \pm([c]+[c]^{-1})$ in $\mathbb{Z}[C]$ from each orbit. By $q$, there are infinitely many $\Sigma_{i=1}^{2k} \pm 2[c]$ in $\mathbb{Z}[A]$.
\end{proof}

\begin{Theorem}\label{Theorem 2.14}
    There exists surjective homomorphism from $MCG(X)$ to $\mathbb{Z}^{\infty} $.
\end{Theorem}

\begin{proof}
     By Lemma 2.13, $\mathbb{Z}^{\infty} \subset Im\Psi'|_{MCG_0(X)}$ in $\mathbb{Z}[A]$. Thus $Im\Psi'$ is also of infinite rank in $\mathbb{Z}[A]$ and $Im\Psi'$ is a surjective homomorphism to $\mathbb{Z}^{\infty}$
\end{proof}

By analogous proof, we have the following corollary:

\begin{Corollary}
There exists surjective homomorphism from $MCG(X')$ to $\mathbb{Z}^{\infty} $.
\end{Corollary}

\section{For non-orientable $S^2$ bundles}

In this section we treat the case of nonorientable $S^2$-bundles over surfaces, imposing no orientability hypothesis on the base surface. We show that the surjectivity property established in the preceding section persists in this more general, nonorientable setting. To avoid ambiguous, nonorientable $S^2$ bundle is defined to be a $S^2$ bundle whose orientation of fiber is not preserved. It is equivelent to say that we extend $Diff^+(S^2)$ to $Diff(S^2)$,following the clutching construction in section 1.

Notice that the following results only works for $g>1$. We do not treat $g=1$ for nonorientable base surfaces. In this case
$\Sigma'=\mathbb{RP}^2$ and $\pi_1(X)\cong\pi_1(\mathbb{RP}^2)\cong\mathbb Z/2$ is
finite, so $\zeta=(\pi_1(X)\setminus\{1\})/\mathrm{conj}$ has a single
element and the Dax invariant cannot
produce a surjection onto $\mathbb Z^\infty$.

\begin{lemma}
Let $\pi:E\to\Sigma$ be an $S^2$-bundle over a connected surface $\Sigma$
(orientable or not). Suppose $E\to\Sigma$ is nonorientable. Then there is
a unique connected double cover $q:\Sigma'\to\Sigma$ such that the pullback bundle $q^*E\to\Sigma'$ is orientable. Moreover, $q^*E\cong\Sigma'\times S^2$ is trivial.
\end{lemma}

\begin{proof}
For $x\in\Sigma$, let $E_x=\pi^{-1}(x)\cong S^2$. The groups
$H_2(E_x;\mathbb Z)\cong\mathbb Z$ form a local system over $\Sigma$, whose
monodromy is a homomorphism
\[
  \delta:\pi_1(\Sigma,x_0)\longrightarrow \mathrm{Aut}\big(H_2(E_{x_0};\mathbb Z)\big)=\{\pm1\}\cong\mathbb Z_2 ,
\]
recording whether transport along a loop preserves the fiber orientation.
The bundle $E\to\Sigma$ is orientable (its structure group reduces to
$\mathrm{Diff}^+(S^2)$) if and only if $\delta$ is trivial, so by
hypothesis $\delta\neq 0$. This construction is natural: for any covering
$q:\Sigma'\to\Sigma$, the bundle $q^*E$ has fibers $(q^*E)_y=E_{q(y)}$, so
its fiber-orientation monodromy is
\[
  \delta_{q^*E}=\delta\circ q_*:\pi_1(\Sigma')\to\mathbb Z_2 .
\]

Since $\delta\neq0$, the subgroup
$H:=\ker\delta\le\pi_1(\Sigma,x_0)$ has index $2$. Let
$q:\Sigma'\to\Sigma$ be the connected double cover with
$q_*\pi_1(\Sigma')=H$. Then $\delta_{q^*E}=\delta\circ q_*=\delta|_H\equiv0$,
so $q^*E$ is orientable. This proves the existence.

For uniqueness, let $q':\Sigma''\to\Sigma$ be a connected double cover
with $(q')^*E$ orientable, and let $H'=q'_*\pi_1(\Sigma'')$, which has
index $2$. Then $(q')^*E$ is orientable iff $\delta|_{H'}\equiv0$, i.e.\
iff $H'\subseteq\ker\delta=H$. Since $[\pi_1(\Sigma):H']=[\pi_1(\Sigma):H]=2$,
this forces $H'=H$. By the classification of connected covers of $\Sigma$
by subgroups of $\pi_1(\Sigma)$, $q'$ is isomorphic to $q$.
\end{proof}

\begin{remark}
In fact $q^*E\cong\Sigma'\times S^2$ is trivial bundle.
Since $\mathrm{Diff}(S^2)\simeq O(3)$ by Smale's, $E=S(V)$ is the sphere
bundle of a rank-$3$ vector bundle $V\to\Sigma$, and rank-$3$ bundles over
a surface are classified by $(w_1,w_2)$. For $q^*V$:
\begin{itemize}
\item $w_1(q^*V)=\delta\circ q_*=\delta|_{\ker\delta}=0$ by the choice of $q$;
\item $w_2(q^*V)=q^*w_2(V)=0$: $\langle q^*w_2(V),[\Sigma']\rangle=\langle w_2(V),q_*[\Sigma']\rangle=0$,
so $q^*w_2(V)=0$ because $H^2(\Sigma';\mathbb Z_2)\cong\mathbb Z_2$ is
detected by evaluation on $[\Sigma']$. If $\Sigma$ is not closed, then
$H^2(\Sigma';\mathbb Z_2)=0$.
\end{itemize}
Thus $q^*V$ is trivial, and so is $q^*E=S(q^*V)$. Notice that the base surface can be orientable or nonorientable for $E$ based on our construction.

\end{remark}

\[
\begin{tikzcd}[column sep=large, row sep=large]
MCG(E) \arrow[r,"L_{2}"] \arrow[rrr,bend left=25,"\Psi_E"'] &
MCG(q^*E)
\arrow[r,"Dax"] &
\mathbb{Z}[\zeta] \arrow[r,"q"] &
\mathbb{Z}[A]
\end{tikzcd}
\]
\begin{Corollary}
    If $E$ has nonorientable base surface, there is a surjective homomorphism from $MCG_0(E)$ to $\mathbb{Z}^{\infty}$.
\end{Corollary}

\begin{proof}
    We need to lift $E$ again to obtain an orientable double cover with orientable base surface. We recall the construction in Corollary \ref{C2.10}.
\end{proof}

\begin{Corollary}
    If $E$ has orientable base surface, there is a surjective homomorphism from $MCG_0(E)$ to $\mathbb{Z}^{\infty}$.
\end{Corollary}

\begin{proof}
    We want to show that the restriction of $\Psi_E$ on $MCG_0(E)$ is such a surjective homomorphism. The only part we need to show is that $q\circ f = q$ $\forall f \in MCG_0(E)$ so the lift of $f$ exists in $MCG_0(q*E)$ and the rest follows Theorem \ref{02k}. This is straightforward since $f \simeq Id$.

\end{proof}

An annalogous situation that we have been through before while facing $MCG$ instead is to check whether the lift under $L_2$ preserves fiber, i.e. $f\circ q = q$ $\forall f \in MCG(E)$. The most straightforward way is to recall the construction in Theorem \ref{Theorem 2.11} and let $L_{2m} = L_{2k}\circ L_2$, $m=2k$ and we have $(f\circ p_{2m})(\pi_1(X^{2m})) =(f\circ p_{2k} \circ q)_*(\pi_1(X^{2m})) = (p_{2k})_*(q_*(\pi_1(X^{2m})) = {p_{2m}}_*(\pi_1(X^{2m}))$:

\[
\begin{tikzcd}[column sep=large, row sep=large]
MCG(E) \arrow[r,"L_{2m}"] \arrow[rrr,bend left=25,"\Psi_E'"'] &
MCG(M)
\arrow[r,"Dax"] &
\mathbb{Z}[\zeta] \arrow[r,"q"] &
\mathbb{Z}[A]
\end{tikzcd}
\]

\begin{Corollary}
    There is a surjective homomorphism from $MCG(E)$ to $\mathbb{Z}^{\infty}$.
\end{Corollary}

\begin{proof}
We show that restriction on $MCG_0(E)$ has image containing $\mathbb{Z}^{\infty}$. The rest follows from Theorem \ref{02k} and \ref{Theorem 2.14}.

\end{proof}

%\section{The mapping class group of $\Sigma\times S^2$}

%\section{Classification of surfaces with a dual sphere}

\section*{Acknowledgement}
The author would like to express her deepest gratitude to Professor Boyu Zhang for his guidance, support, and encouragement throughout the preparation of this work. The author also thanks ChatGPT for the blue graphs.


\begin{thebibliography}{pap-111}

\bibitem[BG21]{BG21}Ryan Budney and David Gabai, Knotted 3-balls in $S^4$, arxiv preprint: https://arxiv.org/abs/1912.09029

\bibitem[F82]{F82} Michael Freedman. The topology of four-dimensional manifolds. J. Differential Geom., 17(3):357–
453, 1982.

\bibitem[LWXZ]{LWXZ} J. Lin, W. Wu, Y. Xie and B. Zhang Dax invariant for closed embedded surface and the mapping class group of $\Sigma \times S^2$, arxiv preprint: https://arxiv.org/abs/2501.16083

\bibitem[GS]{GS} Kirby Clauclus and 4-manifold, Robert E. Gompf, Andras I. Stipsicz

\bibitem[OP]{OP}Orson, Patrick; Powell, Mark. Mapping class groups of simply connected 4-manifolds with boundary.
J. Differ. Geom. 131, No. 1, 199-275 (2025)

\bibitem[P86]{P86}B. Perron. Pseudo-isotopies et isotopies en dimension quatre dans la cat´egorie topologique. Topology,
25(4):381–397, 1986.
\bibitem[Q86]{Q86} Frank Quinn. Isotopy of 4-manifolds. J. Differential Geom., 24(3):343–372, 1986

\bibitem[You26]{You26}
Q.~You,
\emph{On the Dax invariants of $S^2$-bundles over surfaces},
arXiv:2607.04695, 2026.


\end{thebibliography}
\end{document}